\documentclass[12pt]{amsart}
\usepackage[english]{babel}
\usepackage{amssymb,amscd}
\usepackage{color}

\usepackage[all]{xy}

\usepackage{euscript}
\usepackage{amsmath}
\usepackage{ifpdf}

\usepackage{amsfonts}
\usepackage{mathrsfs}

\usepackage{amsmath}

\newtheorem{Theorem}{Theorem}[section]
\newtheorem{Lemma}[Theorem]{Lemma}
\newtheorem{Corollary}[Theorem]{Corollary}
\newtheorem{Proposition}[Theorem]{Proposition}

\newtheorem{Remark}[Theorem]{Remark}
\newtheorem{Example}[Theorem]{Example}

\newtheorem{Definition}[Theorem]{Definition}
\newtheorem{notation}[Theorem]{Notation}
\newtheorem{Question}[Theorem]{Question}
\newtheorem{Conjecture}[Theorem]{Conjecture}
\newcommand{\M}{\mathfrak{m}}

\def\sqr#1#2{{\vcenter{\hrule height.#2pt
\hbox{\vrule width.#2pt height#1pt \kern#1pt \vrule width.#2pt}
\hrule height.#2pt}}}

\def\qed{\hspace*{\fill} $\square$}

\begin{document}

\title[Bounds on Dao numbers and applications]{Bounds on Dao numbers and applications to regular local rings}

\author[A. Ficarra]{Antonino Ficarra}
\address{Department of Mathematics and Computer Sciences, Physics
and Earth Sciences, University of Messina, Viale Ferdinando Stagno d’Alcontres 31, 98166 Messina, Italy.}
\email{antficarra@unime.it}

\author[C. B. Miranda-Neto]{Cleto B.~Miranda-Neto}
\address{Departamento de Matem\'atica, Universidade Federal da
Para\'iba, 58051-900 Jo\~ao Pessoa, PB, Brazil.}
\email{cleto@mat.ufpb.br}

\author[D. S. Queiroz]{Douglas S.~Queiroz}
\address{Departamento de Matem\'atica, Universidade Federal da
Para\'iba, 58051-900 Jo\~ao Pessoa, PB, Brazil.}
\email{douglassqueiroz0@gmail.com}


\begin{abstract} The so-called Dao numbers are a sort of measure of the asymptotic behaviour of full properties of certain product ideals in a Noetherian local ring $R$ with infinite residue field and positive depth. In this paper, we answer a question of H. Dao on how to bound such numbers. The auxiliary tools range from Castelnuovo-Mumford regularity of appropriate graded structures to reduction numbers of the maximal ideal. In particular, we substantially improve previous results (and answer questions) by the authors. Finally, as an application of the theory of Dao numbers, we provide new characterizations of when $R$ is regular; for instance, we show that this holds if and only if the maximal ideal of $R$ can be generated by a $d$-sequence (in the sense of Huneke) if and only if the third Dao number of any (minimal) reduction of the maximal ideal vanishes.
\end{abstract}

\maketitle

\section{Motivation: Dao's problem on the fullness of certain ideals}\label{prelim}

Throughout this note, by a {\it ring} we mean a commutative, Noetherian, unital ring. Let $R$ be either a local ring with residue field $K$ and maximal ideal $\M$, or a standard graded algebra over a field $K$ having a unique graded maximal ideal $\M$. We will assume throughout that $K$ is infinite and $\mathrm{depth} \ R > 0$ (i.e., $\M$ contains a non-zerodivisor), and in addition $I \subset R$ stands for an ideal which we assume to be homogeneous whenever $R$ is graded.

In this paper we focus on  the properties of $\M$-fullness, fullness, and weak $\M$-fullness (to be recalled in the next section) of certain ideals. More precisely,  we are interested in the so-called Dao numbers of the given ideal $I$, i.e., three non-negative integers ${\mathfrak d}_i(I)$,\, 
$i=1, 2, 3$, defined below, which in some sense provide a measure for the asymptotic behaviour of the full properties of certain product ideals involving $I$. 

\begin{Definition}\rm The \textit{Dao numbers} of $I$ are defined as:
\begin{eqnarray}
{\mathfrak d}_1(I) & = & {\rm min} \{t \geq 0 \mid \mbox{$I\M^{k}$  is $\M$-full  for  all  $k \geq t$\}}; \nonumber \\
{\mathfrak d}_2(I) & =  & {\rm min} \{t \geq 0 \mid \mbox{$I\M^{k}$  is full  for  all  $k \geq t$\}}; \nonumber \\
{\mathfrak d}_3(I)  & = & {\rm min} \{t \geq 0 \mid \mbox{$I\M^{k}$  is weakly $\M$-full  for  all  $k \geq t$\}}. \nonumber
\end{eqnarray}

\end{Definition}

It is worth observing that, if $(R, \M, K)$ and $I$ are as above, then as shown in \cite[Proposition 2.2]{CD} the basic relations among the Dao numbers of $I$ are 
$${\mathfrak d}_2(I) \leq {\mathfrak d}_1(I)= {\mathfrak d}_3(I).$$

Our motivation is the following problem suggested by H. Dao, which was first addressed in \cite{CD} and then in \cite{Ficarra}. 

\begin{Question}{\rm (}\cite[Question 4.5]{Dao}{\rm )} \label{qdao} \rm Can we find good lower and upper bounds for the ${\mathfrak d}_{i}(I)$'s? 
\end{Question}

In the case where $I$ is a reduction of $\M$, a lower bound for ${\mathfrak d}_{3}(I)$ is given by ${\rm r}_I(\M)$, the reduction number of $\M$ with respect to $I$, which follows immediately from \cite[Theorem 3.4]{CD} (the question as to whether ${\rm r}_I(\M)\leq {\mathfrak d}_{2}(I)$ remains unanswered). Otherwise, for a general $I$ (not necessarily a reduction of $\M$), a more elaborated lower bound was established in \cite[Proposition 1.8]{Ficarra}.

So, in the present paper, our main goal is to answer the upper bound part of Dao's question, by using two fundamental numerical invariants in commutative algebra: the Castelnuovo-Mumford regularity (of appropriate graded structures) and, again, the reduction number. The connection between these numbers and Dao's problem was first exploited in \cite{CD} and later developed even further (concerning specifically the Castelnuovo-Mumford regularity) in \cite{Ficarra}. The present work establishes, in fact, generalizations and substantial improvements of the results given in these two papers. Additionally, we shall finally derive new characterizations of regular local rings.

\section{Auxiliary notions and basic properties}

In this section, we invoke some basic concepts and facts which we shall freely use in this note (without explicit mention).

\subsection{Full properties of ideals} Let $K$ be an infinite field and $(R, \M)$ be either a local ring with residue field $K$ or a standard graded $K$-algebra having a unique homogeneous maximal ideal $\M$. Assume  $\mathrm{depth} \ R > 0$, and let $I \subset R$ stand for an ideal (homogeneous whenever $R$ is graded).

\begin{Definition}\rm The following notions are central in this paper:
\begin{enumerate}
    \item[(a)] $I$ is {\it $\M$-full} if $I\M : x = I$ for some element $x \in \M \setminus {\M}^2$;
    \item[(b)] $I$ is {\it full} if $I : x = I : \M$ for some element $x \in \M \setminus {\M}^2$;
    \item[(c)] $I$ is {\it weakly $\M$-full} if $I\M : \M = I$.
\end{enumerate}
\end{Definition}

For completeness, we recall a few interesting properties. It is clear that $\M$-full ideals are weakly $\M$-full. If $I$ is $\M$-primary, then $I$ is weakly $\M$-full if and only if $I$ is basically full in the sense of \cite{HRR}. Moreover, $\M$-full ideals satisfy the so-called {\it Rees property}, and if $R$ is a normal domain then any integrally closed ideal is $\M$-full;  see \cite{Watanabe1} (also \cite{GH}).

\subsection{Castelnuovo-Mumford regularity} Let $S = \bigoplus_{m \geq 0}S_{m}$ be a finitely generated standard graded algebra over a ring $S_{0}$. As usual, by {\it standard} we mean that $S$ is generated by $S_1$ as an $S_0$-algebra. We write $S_{+} = \bigoplus_{m \geq 1}S_{m}$ for the ideal generated by all elements of $S$ of positive degree. 
For a graded $S$-module $A=\bigoplus_{m \in {\mathbb Z}}A_{m}$ satisfying $A_m=0$ for all $m\gg 0$, we let $a(A) = \textrm{max}\{m\in {\mathbb Z} \ | \ A_{m} \neq 0\}$ if $A\neq 0$, and $a(A)=-\infty$ if $A=0$.

Now, for a finitely generated graded $S$-module $N\neq 0$ and an integer $j\geq 0$, we take $A=H_{S_{+}}^{j}(N)$
and use the notation $$a_{j}(N) \, := \, a(H_{S_{+}}^{j}(N)),$$
where $H_{S_{+}}^{j}(-)$ stands for the $j$-th local cohomology functor with respect to the ideal $S_{+}$. It is known that $H_{S_{+}}^{j}(N)$ is a graded module with $H_{S_{+}}^{j}(N)_n=0$ for all $n\gg 0$ (see, e.g., \cite[Proposition 15.1.5(ii)]{B-S}). Thus, $a_{j}(N)<+\infty$.

\begin{Definition}\rm Maintain the above setting and notations. The \textit{Castelnuovo-Mumford regularity} of $N$ is defined as
$$\mathrm{reg}_SN \, := \, \mathrm{max}\{a_j(N) + j \, \mid \, j \geq 0\}.$$
\end{Definition}

It is well-known that $\mathrm{reg}\,N$ governs the complexity of the graded structure of $N$ and is relevant in commutative algebra and algebraic geometry, for example in the study of degrees of syzygies over polynomial rings (see, e.g., \cite[Chapter 15]{B-S}).

\begin{Remark}\rm A classical instance of interest is when $S=\mathcal{R}(J)$, the Rees algebra of an ideal $J$ in a ring $R$ (to be recalled in the next subsection), which is known to be a finitely generated standard graded $R$-algebra. In particular, we can consider the case where $R$ is local and $J=\M$, the maximal ideal of $R$.
\end{Remark}

We recall below a few basic rules about this invariant. For details, we refer to  \cite[p.\,277, (a), (c) and (d)]{B-C-R-V}, \cite[Corollary 20.19]{Eisenbud} and \cite[Lemma 3.1]{HT}.

\smallskip
\begin{itemize}
    \item[(i)] As usual, given an integer $j$, we denote by $N(j)$ the module $N$ with degrees shifted by $j$, that is, $N(j)_i = N_{i + j}$ for all $i$. Then, $$\mathrm{reg}_{S}N(j) = \mathrm{reg}_{S}N - j.$$

    \item[(ii)] Let $0 \rightarrow M \rightarrow N \rightarrow P \rightarrow 0$ be a short exact sequence of finitely generated graded $S$-modules. Then: 
    \begin{itemize}
        \item[$\bullet$] $\mathrm{reg}_{S}N \leq \mathrm{max}\{\mathrm{reg}_{S}M, \mathrm{reg}_{S}P\},$ with equality if $\mathrm{reg}_{S}P\neq \mathrm{reg}_{S}M - 1$ or  $M_{k} = 0$ for $k \gg 0$.
        \item[$\bullet$] $\mathrm{reg}_{S}M \leq \mathrm{max}\{\mathrm{reg}_{S}N, \mathrm{reg}_{S}P + 1\},$ with equality if $\mathrm{reg}_{S}N \neq \mathrm{reg}_{S}P.$ 
        \item[$\bullet$] $\mathrm{reg}_{S}P \leq \mathrm{max}\{\mathrm{reg}_{S}N, \mathrm{reg}_{S}M - 1\},$ with equality  if $\mathrm{reg}_{S}N\neq \mathrm{reg}_{S}M.$
    \end{itemize}

\smallskip

\item[(iii)] If $N_j = 0$ for all $j \gg 0$, then $\mathrm{reg}_{S}N = a(N)$.
\end{itemize}

\subsection{Rees algebra and Dao module}

Here, we consider some useful graded structures. Let $R$ be a ring and $J$ an ideal of $R$. 

\begin{Definition}\rm The {\it Rees algebra} of $J$ is the graded ring $$\mathcal{R}(J) =
\bigoplus_{k \geq 0} J^k = R\oplus J\oplus J^2 \oplus \cdots$$ 
    
\end{Definition}

\begin{notation}\label{extended}\rm  Given an $R$-module $M$, it is customary to write $\mathcal{R}(J, M)=\bigoplus_{k \geq 0} J^kM$ for the Rees module of $J$ relative to $M$. In this paper, if $I$ is another $R$-ideal, we will be particularly interested in $$\mathcal{R}(J, I) =
\bigoplus_{k \geq 0} IJ^k.$$ Note $\mathcal{R}(J, I) =I\mathcal{R}(J)$, the
extension of $I$ to the ring $\mathcal{R}(J)$, which is therefore a finitely generated ideal of $\mathcal{R}(J)$. \end{notation}

Now let $(R, \M)$ be as in Section \ref{prelim}. Its associated graded ring  is defined by $\mathrm{gr}_{\M}(R) = \bigoplus_{k \geq 0}\M^{k}/\M^{k + 1}$. In \cite{Ficarra}, the following graded structure is introduced for a given ideal $I$ of $R$.

\begin{Definition}\rm The \textit{Dao module} of $I$ is given by $$\mathfrak{D}_{\M}(I) = \bigoplus_{k \geq 0} \frac{I\M^{k + 1} : \M}{I\M^{k}},$$ which is a graded $\mathcal{R}(\M)$-module.
\end{Definition}

\begin{Remark}\rm The $k$th component of the Dao module vanishes if and only if the ideal $I{\M}^k$ is weakly $\M$-full. Since $$\mathfrak{D}_{\M}(I)_k = 0 \quad \mbox{for\, all} \quad
k \geq {\mathfrak d}_3(I),$$ it follows that $\mathfrak{D}_{\M}(I)$ has finite length and therefore is a finitely
generated graded $\mathcal{R}(\M)$-module, satisfying $$\mathrm{reg}_{\mathcal{R}(\M)}\,\mathfrak{D}_{\M}(I) ={\mathfrak d}_3(I) - 1$$ whenever ${\mathfrak d}_3(I)\geq 1$ (e.g., if $I$ is not weakly $\M$-full).
\end{Remark}

\subsection{Ratliff-Rush operation} Let $I$ be an ideal of a ring $R$.

\begin{Definition}\rm The \textit{Ratliff-Rush closure} $\widetilde{I}$
of the ideal $I$ is given by
$$\widetilde{I}  =  \bigcup_{m \geq 1}\,I^{m+1} : I^{m}.$$ \end{Definition}

This is an ideal of $R$ containing $I$ which in fact refines the integral closure of $I$, so that $\widetilde{I}=I$ whenever $I$ is integrally closed. For details, see
\cite{Ratliff-Rush}. 

Now suppose $I$ contains a regular element, i.e., a non-zerodivisor on $R$. Then it is well-known that $\widetilde{I}$ is the largest ideal that shares with $I$ the same sufficiently high powers; hence, $$\widetilde{I^{m}} = I^{m} \quad  \mbox{for\, all} \quad m \gg 0.$$ This enables us to consider the following helpful number (inspired by \cite[Proposition 4.2]{Rossi-Swanson}).

\begin{notation}\rm
If $I$ contains a regular element, we set $$s(I)  =  \mathrm{min}\,\{n \ge1 \, \mid \, \widetilde{I^{i}} = I^{i} \ \mathrm{for \ all} \ i \geq n \}.$$
\end{notation}

\begin{Remark}\label{Rem:depth}\rm Let us invoke a couple of useful properties. First, according to \cite[Lemma 2.2]{Mafi} we can write $$\widetilde{I^{k + 1}} : I = \widetilde{I^{k}} \quad \mbox{for \,all} \quad k \geq 0 .$$ Moreover, if $\mathrm{gr}_{I}(R) = \bigoplus_{k \geq 0}I^{k}/I^{k + 1}$ denotes the associated graded ring of $I$, then by \cite[Remark 1.6]{Rossi-Swanson} we get that $\widetilde{I^{k}} = I^{k}$ for all $k\geq 0$ (i.e., $s(I)=1$) if and only if $$\mathrm{depth}\,\mathrm{gr}_{I}(R) > 0.$$ \end{Remark}

\subsection{Reduction number} One last auxiliary notion is in order. 

\begin{Definition}\rm Let $J$ be a proper ideal of a ring $R$. An ideal $I \subset J$ is said to be a \textit{reduction of $J$} if $IJ^{r} = J^{r + 1}$ for some integer $r\geq 0$. Such a reduction $I$ is \textit{minimal} if it is minimal with respect to inclusion.  If $I$ is a reduction of $J$, we define the \textit{reduction number of $J$ with respect to $I$} as the number
$${\rm r}_{I}(J) =  \mathrm{min}\,\{m \in \mathbb{N} \, \mid \, IJ^{m} = J^{m + 1}\},$$ and the \textit{reduction number of $J$} as
$${\rm r}(J) =  \mathrm{min}\,\{{\rm r}_{I}(J) \, \mid \, \mbox{$I$ is a minimal reduction of $J$}\}.$$ \end{Definition}

Of special interest in this paper will be the case where $(R, \M)$ is a local ring and $J=\M$.

\section{Upper bounds on Dao numbers via Castelnuovo-Mumford regularity}

Before establishing the results of this section, we fix a piece of notation. For a graded $R$-module $M = \bigoplus_{k \geq 0} M_k$ and an integer $\ell \geq 0$, we can consider the truncation $M_{\geq \ell} = \bigoplus_{k \geq \ell} M_k$.
Specifically, using Notation \ref{extended} and assuming $(R,\M,K)$ is a local ring or a standard graded $K$-algebra, we will be interested in the truncation $$\mathcal{R}(\M, I)_{\geq 1}=
\bigoplus_{k \geq 1} I{\M}^k.$$

Here we are focused on tackling the upper bound part of Question \ref{qdao} in terms of the Castelnuovo-Mumford regularity of appropriate graded structures. The first result in this direction, in  case $R$ is local and $I$ is a reduction of $\M$, was proved in \cite{CD} and can be stated as follows.

\begin{Theorem}$($\cite[Theorem 3.10]{CD}$)$\label{CD-reg}
	Let $(R, \M, K)$ be a local ring with infinite residue field and ${\rm depth}\,R>0$, and let $I$ be a reduction of $\M$. Then, 
	$${\mathfrak d}_3(I) \leq \mathrm{reg}_{\mathcal{R}(\M)}\,\mathcal{R}(\M).$$
	
\end{Theorem}

Later, in \cite{Ficarra}, the following general answer to Question \ref{qdao} was provided.

\begin{Theorem}$($\cite[Theorem 1.1]{Ficarra}$)$\label{antonino1.1}
    Let $(R,\M,K)$ be either a local ring or a standard graded $K$-algebra, with $K$ infinite and $\mathrm{depth}\,R > 0$. Let $I \subset R$ be an ideal $($homogeneous if $R$ is graded$)$. Then, $${\mathfrak d}_3(I) \leq \mathrm{max}\{\mathrm{reg}_{\mathcal{R}(\M)}\,\mathcal{R}(\M, I),\ 
 \mathrm{reg}_{\mathcal{R}(\M)}\,\mathcal{R}(\M, I)_{\geq 1} :_{\mathcal{R}(R)} \M\}.$$ 
\end{Theorem}

\medskip

We present the following auxiliary basic lemma, which  additionally refines \cite[Proposition 1.6]{Ficarra} by establishing the relationship between the regularity of an extended ideal as in Notation \ref{extended} and that of the ambient Rees algebra $\mathcal{R}(J)$.

\begin{Lemma}\label{lemareg}
 Let $(R, \M , K)$ be either a local ring or a standard graded $K$-algebra. Let $I, J$ be ideals of $R$. Then,
$$\mathrm{reg}_{\mathcal{R}(J)}\mathcal{R}(J) \leq \mathrm{reg}_{\mathcal{R}(J)} \mathcal{R}(J,I),$$ with equality if $I$ is a reduction of $J$.
\end{Lemma}
\begin{proof}
    Since $\mathcal{R}(J,I)$ is a homogeneous ideal of $\mathcal{R}(J)$, the short exact sequence $$0 \rightarrow \mathcal{R}(J,I) \rightarrow \mathcal{R}(J) \rightarrow \frac{\mathcal{R}(J)}{\mathcal{R}(J,I)} \rightarrow 0$$ yields \begin{eqnarray}
        \mathrm{reg}_{\mathcal{R}(J)}\mathcal{R}(J) &\leq& \mathrm{max}\Big\{\mathrm{reg}_{\mathcal{R}(J)} \mathcal{R}(J,I), \mathrm{reg}_{\mathcal{R}(J)} \frac{\mathcal{R}(J)}{\mathcal{R}(J,I)} \Big\} \nonumber \\ ~ &=& \mathrm{max}\Big\{\mathrm{reg}_{\mathcal{R}(J)} \mathcal{R}(J,I), \mathrm{reg}_{\mathcal{R}(J)} \mathcal{R}(J,I) - 1 \Big\} \nonumber \\ ~ & = & \mathrm{reg}_{\mathcal{R}(J)} \mathcal{R}(J,I). \nonumber
    \end{eqnarray} For the equality part, the proof of $\mathrm{reg}_{\mathcal{R}(J)}\mathcal{R}(J) \geq \mathrm{reg}_{\mathcal{R}(J)} \mathcal{R}(J,I)$ is exactly the same as the one given in \cite[Proposition 1.6]{Ficarra} for the case $J=\M$. \qed
\end{proof}

\medskip

Our main result establishes \cite[Conjecture 0.1]{Ficarra} in its full generality, as follows.

\begin{Theorem}\label{antonino1.3}
    Let $R$ and $I$ be as in Theorem \ref{antonino1.1}. Then, $${\mathfrak d}_3(I) \leq \mathrm{reg}_{\mathcal{R}(\M)}\,\mathcal{R}(\M, I).$$ 
\end{Theorem}
\begin{proof} Clearly, we may suppose $\mathfrak{d}_3(I)>0$. Set $t = \mathfrak{d}_3(I) - 1$, and let us assume first that $s(\mathfrak{m}) \leq t$. Then, $$I\mathfrak{m}^{t + 1} : \mathfrak{m} \subseteq \mathfrak{m}^{t + 1} : \mathfrak{m} = \widetilde{\mathfrak{m}^{t + 1}} : \mathfrak{m} = \widetilde{\mathfrak{m}^{t}} = \mathfrak{m}^{t}.$$

Moreover, there is a well-known isomorphism (see, for instance, \cite[p.\,268]{HHBook})
$$
H_n(\mathcal{R}(\M)_+;\mathcal{R}(\M)/\mathcal{R}(\M, I))\cong e_1\wedge\cdots\wedge e_n\,(0:_{\mathcal{R}(\M)/\mathcal{R}(\M,I)}\mathcal{R}(\M)_+),
$$
where the first module is the $n$th Koszul homology module of $\mathcal{R}(\M)_+$ with respect to $\mathcal{R}(\M)/\mathcal{R}(\M, I)$, the integer $n$ is the minimal number of generators of the irrelevant ideal $\mathcal{R}(\M)_+$, and each $e_i$ has degree one. Notice that we can write
\begin{equation}\label{eq:Soc-R-m-I}
        0:_{\mathcal{R}(\M)/\mathcal{R}(\M,I)}\mathcal{R}(\M)_+ \ =\ \bigoplus_{k\ge0}\frac{(I\M^{k+1}:\M)\cap\M^k}{I\M^k},\end{equation}
and furthermore
    \begin{equation}
        (0:_{\mathcal{R}(\M)/\mathcal{R}(\M,I)}\mathcal{R}(\M)_+)_t\ =\ \frac{(I\M^{t+1}:\M)\cap\M^t}{I\M^t}\ =\ \frac{I\M^{t+1}:\M}{I\M^t} \ =\ \mathfrak{D}_\M(I)_t \neq 0.
    \end{equation}   
Consequently, using \cite[Theorem 8.1.3]{B-C-R-V} we obtain
    \begin{align*}
        \mathfrak{d}_3(I)-1\ =\ t\ &\le\ \max\{j \mid (0:_{\mathcal{R}(\M)/\mathcal{R}(\M,I)}\mathcal{R}(\M)_+)_j\ne0\}\\
        &=\ \max\{j-n \mid H_n(\mathcal{R}(\M)_+;\mathcal{R}(\M)/\mathcal{R}(\M, I))_j\ne0\}\\
        &\le\ \max\{j-i \mid H_i(\mathcal{R}(\M)_+;\mathcal{R}(\M)/\mathcal{R}(\M, I))_j\ne0\}\\
        &=\ \textrm{reg}_{\mathcal{R}(\M)}\mathcal{R}(\M)/\mathcal{R}(\M,I)\\
        &=\ \textrm{reg}_{\mathcal{R}(\M)}\mathcal{R}(\M,I)-1.
    \end{align*}

    Next, assume that $t < s(\M)$. Since $\M$ contains a regular element, we deduce from  \cite[Proposition 2.1(ii)]{Rossi-Dinh-Trung} that $s(\M) \leq \mathrm{max}\{1, \mathrm{reg}_{\mathcal{R}(\M)}\,\mathcal{R}(\M)\}$. If $\mathrm{reg}_{\mathcal{R}(\M)}\,\mathcal{R}(\M) = 0$ then $\mathrm{depth}\,\mathrm{gr}_\M(R)>0$
 (see Remark \ref{Rem:depth}), and the result follows from \cite[Theorem 1.3]{Ficarra}. On the other hand, if  $\mathrm{reg}_{\mathcal{R}(\M)}\,\mathcal{R}(\M) \geq 1$, given that $t < s(\M)$, we conclude  $${\mathfrak d}_3(I) \leq s(\M) \leq \mathrm{reg}_{\mathcal{R}(\M)}\,\mathcal{R}(\M) \leq \mathrm{reg}_{\mathcal{R}(\M)}\,\mathcal{R}(\M, I),$$ where the last inequality is guaranteed by Lemma \ref{lemareg}. \qed
\end{proof}

\medskip

\begin{Remark}\rm  Notice that we can rediscover Theorem
\ref{CD-reg} as a direct consequence of Theorem \ref{antonino1.3} and Lemma \ref{lemareg}.
    
\end{Remark}

In our view, and in connection to Lemma \ref{lemareg}, it is worth asking the following.

\begin{Question}\label{quesred}\rm
 Let $(R, \M , K)$ be either a local ring or a standard graded $K$-algebra. Let $I\subset J$ be ideals of $R$. When does the condition
$\mathrm{reg}_{\mathcal{R}(J)}\mathcal{R}(J) = \mathrm{reg}_{\mathcal{R}(J)} \mathcal{R}(J,I)$ force $I$ to be a reduction of $J$?
\end{Question}

Note that the inclusion $$0:_{\mathcal{R}(\M)/\mathcal{R}(\M,I)}\mathcal{R}(\M)_+\ =\ \bigoplus_{k\ge0}\frac{(I\M^{k+1}:\M)\cap\M^k}{I\M^k} \ \subseteq \bigoplus_{k \geq 0} \frac{I\M^{k + 1} : \M}{I\M^{k}} \ = \ \mathfrak{D}_{\M}(I)$$ always holds, and naturally we can ask when it is actually an equality.

\begin{Proposition}
    Let $R$ and $I$ be as in Theorem \ref{antonino1.1}. If $R$ is a standard graded $K$-algebra or $\mathrm{depth}\,\mathrm{gr}_{\M}(R)>0$, then  $0:_{\mathcal{R}(\M)/\mathcal{R}(\M,I)}\mathcal{R}(\M)_+=\mathfrak{D}_\M(I)$.
\end{Proposition}
\begin{proof}
    In view of equation (\ref{eq:Soc-R-m-I}), it suffices to show that
    $$I\M^{k+1}:\M\ \subset\ \M^k \quad \mbox{for\, all} \quad k\geq 0.$$ First, if $R$ is a standard graded $K$-algebra, we can represent it as a quotient $S/J$ where $S=K[x_1,\dots,x_n]$ is a standard graded polynomial ring and $J\subset S$ is a homogeneous ideal. Now let $f\in I\M^{k+1}:\M$ be an homogeneous element. Then $$(x_1+J)f\in I\M^{k+1}\subset\M^{k+1}$$ and so $\deg f+1\ge k+1$. Thus $f\in\M^{k}$ and consequently $I\M^{k+1}:\M\subset\M^k$ for all $k\ge0$. Finally, if $\mathrm{depth}\,\mathrm{gr}_\M(R)>0$ then, as we know, $\widetilde{\M^{k}} = \M^{k}$ for all $k \geq 0$. It follows that $$I\M^{k + 1} : \M \subset \M^{k + 1} : \M = \widetilde{\M^{k + 1}} : \M = \widetilde{\M^{k}} = \M^{k}.$$ \qed
\end{proof}

\medskip

Next, we consider some natural blowup algebras that are related to the Dao module of $I$ and whose Castelnuovo-Mumford regularity will be compared to that of the ideal $\mathcal{R}(\M,I)\subset \mathcal{R}(\M)$.

\begin{Definition}\rm The {\it associated graded module of $\M$ relative to $I$} is the $\mathcal{R}(\M)$-module  $$\mathrm{gr}_{\M}(I)=\mathcal{R}(\M,I)/\M\mathcal{R}(\M,I)=\bigoplus_{k\ge0}I\M^k/I\M^{k+1}.$$
\end{Definition}

Note that, for each $k\ge0$, we have the following commutative diagram with exact rows and columns:
\begin{displaymath}
    \xymatrix {
    &0\ar[d]&0\ar[d]&&\\
    &I\M^{k+1}\ar[d]\ar@{=}[r]&I\M^{k+1}\ar[d]&&\\
    0\ar[r]&I\M^k\ar[d]\ar[r]&I\M^{k+1}:\M\ar[d]\ar[r]&\cfrac{I\M^{k+1}:\M}{I\M^k}\ar[r]\ar@{=}[d]&0\\
    0\ar[r]&\cfrac{I\M^k}{I\M^{k+1}}\ar[d]\ar[r]&\cfrac{I\M^{k+1}:\M}{I\M^{k+1}}\ar[d]\ar[r]&\cfrac{I\M^{k+1}:\M}{I\M^k}\ar[r]&0\\
    &0&0&&
    }
\end{displaymath}

\begin{notation}\rm  For simplicity, set $$\mathcal{P}_\M(I)=\mathcal{R}(\M,I)_{\ge1}:_{\mathcal{R}(R)}\M \quad \mbox{and} \quad \mathcal{Q}_\M(I)=(\mathcal{R}(\M,I)_{\ge1}:_{\mathcal{R}(R)}\M)/\mathcal{R}(\M,I)_{\ge1}.$$
\end{notation}

Taking the direct sum in the above diagram, we obtain the following commutative diagram of finitely generated graded $\mathcal{R}(\M)$-modules with exact rows and columns:
\begin{displaymath}
    \xymatrix {
    &0\ar[d]&0\ar[d]&&\\
    &\mathcal{R}(\M,I)_{\ge1}(1)\ar[d]\ar@{=}[r]&\mathcal{R}(\M,I)_{\ge1}(1)\ar[d]&&\\
    0\ar[r]&\mathcal{R}(\M,I)\ar[d]\ar[r]&\mathcal{P}_\M(I)(1)\ar[d]\ar[r]&\mathfrak{D}_\M(I)\ar[r]\ar@{=}[d]&0\\
    0\ar[r]&\mathrm{gr}_\M(I)\ar[d]\ar[r]&\mathcal{Q}_\M(I)(1)\ar[d]\ar[r]&\mathfrak{D}_\M(I)\ar[r]&0\\
    &0&0&&
    }
\end{displaymath}

The lemma below is a special case of \cite[Corollary 3]{Za}.

\begin{Lemma}\label{Zamani} There is an equality $\mathrm{reg}_{\mathcal{R}(\M)}\mathcal{R}(\M,I)=\mathrm{reg}_{\mathcal{R}(\M)}\mathrm{gr}_\M(I)$.
    
\end{Lemma}

Now we obtain the following lower bound for $\mathrm{reg}_{\mathcal{R}(\M)}\mathcal{R}(\M,I)$.

\begin{Corollary}
    Let $R$ and $I$ be as in Theorem \ref{antonino1.1}. Then,
    $$\mathrm{reg}_{\mathcal{R}(\M)}\mathcal{R}(\M,I)\geq
    {\rm max}\{\mathrm{reg}_{\mathcal{R}(\M)}\mathcal{P}_\M(I),\, \mathrm{reg}_{\mathcal{R}(\M)}\mathcal{Q}_\M(I)\}\ -1.
    $$
    In particular,
    $$
    \mathrm{reg}_{\mathcal{R}(\M)}\mathcal{P}_\M(I)=\mathrm{reg}_{\mathcal{R}(\M)}\mathcal{Q}_\M(I)=\mathrm{reg}_{\mathcal{R}(\M)}\mathcal{R}(\M,I)+1
    $$
    if either $\mathfrak{d}_3(I)=0$ or $\mathrm{reg}_{\mathcal{R}(\M)}\mathcal{R}(\M,I)>\mathfrak{d}_3(I)>0$.
\end{Corollary}
\begin{proof} There are short exact sequences of finitely generated graded $\mathcal{R}(\M)$-modules
    $$
    0 \rightarrow \mathcal{R}(\M, I) \rightarrow \mathcal{P}_\M(I)(1) \rightarrow \mathfrak{D}_{\M}(I) \rightarrow 0,
    $$
    $$
    0 \rightarrow \mathrm{gr}_\M(I) \rightarrow \mathcal{Q}_\M(I)(1) \rightarrow \mathfrak{D}_{\M}(I) \rightarrow 0.
    $$ Recall that Lemma \ref{Zamani} gives $\mathrm{reg}_{\mathcal{R}(\M)}\mathcal{R}(\M,I)=\mathrm{reg}_{\mathcal{R}(\M)}\mathrm{gr}_\M(I)$. 
    
    First, assume that
    $\mathfrak{d}_3(I)=0$. Then, $\mathfrak{D}_\M(I)=0$. In this case, $\mathcal{P}_\M(I)(1)\cong\mathcal{R}(\M,I)$ and $\mathcal{Q}_\M(I)(1)\cong\mathrm{gr}_\M(I)$, so that $$\mathrm{reg}_{\mathcal{R}(\M)}\mathcal{P}_\M(I)-1=\mathrm{reg}_{\mathcal{R}(\M)}\mathcal{R}(\M,I)=\mathrm{reg}_{\mathcal{R}(\M)}\mathrm{gr}_\M(I)=\mathrm{reg}_{\mathcal{R}(\M)}\mathcal{Q}_\M(I)-1.$$

    Note that if $\mathfrak{d}_3(I)>0$ then in particular $\mathrm{reg}_{\mathcal{R}(\M)}\mathfrak{D}_\M(I)=\mathfrak{d}_3(I)-1$. On the other hand, Theorem \ref{antonino1.3} yields $$\mathfrak{d}_3(I)\le\mathrm{reg}_{\mathcal{R}(\M)}\mathcal{R}(\M,I)=\mathrm{reg}_{\mathcal{R}(\M)}\mathrm{gr}_\M(I)$$ and thus the above exact sequence implies 
    \begin{align*}
        \mathrm{reg}_{\mathcal{R}(\M)}\mathcal{P}_\M(I)(1)\ &\le\ \max\{\mathrm{reg}_{\mathcal{R}(\M)}\mathcal{R}(\M,I),\mathrm{reg}_{\mathcal{R}(\M)}\mathfrak{D}_\M(I)\}\\
        &=\ \max\{\mathrm{reg}_{\mathcal{R}(\M)}\mathcal{R}(\M,I),\mathfrak{d}_3(I)-1\}\\
        &=\ \mathrm{reg}_{\mathcal{R}(\M)}\mathcal{R}(\M,I).
    \end{align*}
    Similarly, $\mathrm{reg}_{\mathcal{R}(\M)}\mathcal{Q}_\M(I)(1)\le\mathrm{reg}_{\mathcal{R}(\M)}\mathcal{R}(\M,I)$. Hence, $$\mathrm{reg}_{\mathcal{R}(\M)}\mathcal{P}_\M(I),\mathrm{reg}_{\mathcal{R}(\M)}\mathcal{Q}_\M(I)\le\mathrm{reg}_{\mathcal{R}(\M)}\mathcal{R}(\M,I)+1.$$ So, if now we suppose that $\mathrm{reg}_{\mathcal{R}(\M)}\mathcal{R}(\M,I)>\mathfrak{d}_3(I)>0$, we finally derive $$\mathrm{reg}_{\mathcal{R}(\M)}\mathcal{R}(\M,I)-1\ne\mathrm{reg}_{\mathcal{R}(\M)}\mathfrak{D}_\M(I),$$ and the equalities  $\mathrm{reg}_{\mathcal{R}(\M)}\mathcal{P}_\M(I)=\mathrm{reg}_{\mathcal{R}(\M)}\mathcal{Q}_\M(I)=\mathrm{reg}_{\mathcal{R}(\M)}\mathcal{R}(\M,I)+1$ follow. \qed
\end{proof}

\medskip

Inspired by Theorem \ref{CD-reg} and Theorem \ref{antonino1.3}, we might wonder whether the comparison $${\mathfrak d}_3(I) \leq \mathrm{reg}_{\mathcal{R}(\M)}\,\mathcal{R}(\M)$$ holds in general. However, in the case where $I$ is not a reduction of $\M$, the relationship with the regularity of $\mathcal{R}(\M)$ may become rather wild, as we can see in the examples below.

\begin{Example}\label{powerk} \rm
    Let $(R, \M)$ be a local ring with infinite residue field and ${\rm depth}\,R>0$. By \cite[Proposition 1.5]{Rossi-Dinh-Trung}, we have  $\M^{n + 1} : \M = \M^{n}$ for all $n \geq \mathrm{reg}_{\mathcal{R}(\M)}\,\mathcal{R}(\M)$. Now set $I={\M}^k$ for any given $k \geq 2$ (note $I$ cannot be a reduction of $\M$). We can write \begin{equation}\label{power} I\M^{n} = \M^{n + k} = \M^{n + k + 1} : \M = I\M^{n+1} : \M \end{equation} for all $n \geq \mathrm{reg}_{\mathcal{R}(\M)}\,\mathcal{R}(\M)$ and therefore  ${\mathfrak d}_3(I) \leq  \mathrm{reg}_{\mathcal{R}(\M)}\,\mathcal{R}(\M)$. Furthermore notice that, by using (\ref{power}) whenever $ n \geq \mathrm{reg}_{\mathcal{R}(\M)}\,\mathcal{R}(\M)-k\geq 0$, we obtain  $${\mathfrak d}_3(I) \leq  \mathrm{reg}_{\mathcal{R}(\M)}\,\mathcal{R}(\M)-k<\mathrm{reg}_{\mathcal{R}(\M)}\,\mathcal{R}(\M).$$ Finally, if $k\geq  \mathrm{reg}_{\mathcal{R}(\M)}\,\mathcal{R}(\M)$ then it is easy to see that ${\mathfrak d}_3(I)=0$. In particular, if $\mathrm{reg}_{\mathcal{R}(\M)}\,\mathcal{R}(\M)\leq 2$ then ${\mathfrak d}_3(I)=0$.
\end{Example}

\begin{Example}\rm Let $R$ be as in \cite[Example 4.3]{CD}. As observed there, $\mathrm{reg}_{\mathcal{R}(\M)}\,\mathcal{R}(\M)=8$. Consider the ideal $I={\M}^2$. By Example \ref{powerk} above, we can write  $${\mathfrak d}_3(I) \leq  \mathrm{reg}_{\mathcal{R}(\M)}\,\mathcal{R}(\M)-2=6,$$ whereas,  on the other hand, a computation shows $I{\M}^5\neq I{\M}^6 : \M$. Therefore, we must have ${\mathfrak d}_3(I)=6$.

\end{Example}

We are also able to illustrate the
opposite situation, as follows.

\begin{Example} \rm
     Let $I = (x^{a}, y^{a}) \subset R=k[\![x,y]\!]$, where $k$ is a field and $a \geq 2$. Clearly, $I$ is not a reduction of $\M=(x, y)$. From \cite[Example 4.1]{Dao} we can write  $${\mathfrak d}_3(I) = a - 1 > 0 = \mathrm{reg}_{\mathcal{R}(\M)}\,\mathcal{R}(\M),$$ where the last equality holds because $R$ is a regular local ring (see Theorem \ref{charact2} for a more general statement).
\end{Example}

\section{Approach (and a conjecture) via reduction numbers}

We begin recalling the following result in dimension one.

\begin{Proposition}$($\cite[Corollary 3.7]{CD}$)$ If $(R, \M)$ is a one-dimensional Cohen-Macaulay local ring with infinite residue field, then $$\mathfrak{d}_3(I) = {\rm r}(\M)$$ for any minimal reduction $I$ of $\M$.
    
\end{Proposition}

It is worth pointing out that this proposition is no longer true if the reduction $I$ is not required to be minimal, as exemplified in \cite[Example 4.2]{CD}.

In higher dimension, there is the conjecture below.

\begin{Conjecture} $($\cite[Conjecture 3.8]{CD}$)$\label{CD-conject}  If $(R, \M)$ is a Cohen-Macaulay local ring with infinite residue field and
${\rm dim}\,R \geq 2$, then $$\mathfrak{d}_3(I) = {\rm r}_I(\M)$$ for any minimal reduction $I$ of $\M$. 

\end{Conjecture}

Our main purpose in this section is to provide partial answers to this conjecture. A crucial tool in this investigation is given by the following result.

\begin{Theorem} $($\cite[Theorem 3.4]{CD}$)$ \label{CD-thm} Let $(R, \M)$  be a local ring with infinite residue field and ${\rm depth}\,R>0$, and let $I$ be a reduction of $\M$. Then 
	$$\mathfrak{d}_3(I)  = \mathrm{max}\{{\rm r}_{I}(\M), s(\M) - 1\}.$$
    
\end{Theorem}

Now we can illustrate Conjecture \ref{CD-conject} in the case ${\rm dim}\,R=2$.

\begin{Example}\rm  Let $(R, \M)$ be the local ring of a rational triple point as in \cite[Example 4.1]{CD}, where we highlighted that $s(\M)=1$. Then, by Theorem \ref{CD-thm}, we obtain $\mathfrak{d}_3(I)  = {\rm r}_{I}(\M)$ for any reduction (in particular, minimal reduction) $I$ of $\M$.
    
\end{Example}

The above example actually shows that Conjecture \ref{CD-conject} is true in the case $s(\M) = 1$, or equivalently, ${\rm depth}\,{\rm gr}_{\M}(R)>0$ (see also Proposition \ref{minmult} below).  Furthermore, it is clear that $\widetilde{\M} = \M$, which forces $s(\M) \neq 2$. Therefore, in tackling the conjecture we can suppose $s(\M) \geq 3$.

\begin{Theorem}\label{teoconjecture} Conjecture \ref{CD-conject} holds true in case
 $s := s(\M)\geq 3$ and $$I\widetilde{\M^{s - 2}} = I\M^{s - 2}.$$
\end{Theorem}
\begin{proof} By virtue of Theorem \ref{CD-thm}, it suffices to show that ${\rm r}_{I}(\M) \geq s - 1$. Suppose, by way of contradiction, that $${\rm r}_{I}(\M) \leq s - 2.$$ Since $I$ is a minimal reduction of $\M$, we can apply \cite[Proposition 2.4 and Theorem 2.10]{Mafi} to obtain $I\widetilde{\M^k} = \widetilde{\M^{k + 1}}$ for all $k \geq {\rm r}_{I}(\M)$. Consequently,
    \begin{equation}
        \widetilde{\M^{s - 1}} = I\widetilde{\M^{s - 2}} = I\M^{s - 2} = \M^{s - 1}, \nonumber
    \end{equation}
    which contradicts the definition of $s$.   \qed 
\end{proof} 

\medskip

Since $\M = \widetilde{\M}$ is always true, the case  $s(\M) = 3$ is a straightforward consequence of Theorem \ref{teoconjecture}.

\begin{Corollary}\label{cors=3}
    Conjecture \ref{CD-conject} is true if $s(\M) = 3$.
\end{Corollary}

\begin{Example}\rm
    Let $K$ be an infinite field and $$R = \frac{K[\![x_1,x_2,x_3,x_4,x_5,x_6,x_7]\!]}{(x_1^{2},x_1x_2,x_1x_3,x_1x_4,x_2x_3,x_2x_4,x_3x_4,x_2^{3} - x_1x_5,x_3^{3} - x_1x_6,x_4^{3} - x_1x_7)},$$ which is a three-dimensional Cohen-Macaulay local ring. Notice that $$x_1 \in \widetilde{\M^2} \setminus \M^2$$ and $\widetilde{\M^n}=\M^n$ for all $n \geq 3$. Thus, $s(\M) = 3$. By  Corollary \ref{cors=3}, we obtain $\mathfrak{d}_3(I) = {\rm r}_I(\M)$ for any minimal reduction $I$ of $\M$.
\end{Example}

When $s:=s(\M) \geq 4$, it may be difficult to determine whether $I\widetilde{\M^{s - 2}} = I\M^{s - 2}$. However, by employing a different argument, we are able to establish Conjecture \ref{CD-conject} in the case $s(\M) = 4$.

\begin{Proposition}
    Conjecture \ref{CD-conject} is true if $s(\M) = 4$.
\end{Proposition}
\begin{proof}
    Suppose  ${\rm r}_{I}(\M) \leq s(\M) - 2 = 2$. Then, by \cite[Theorem 3.6]{C-P-R}, the ring $\mathrm{gr}_{\M}(R)$ must be Cohen-Macaulay. Hence ${\rm depth}\,\mathrm{gr}_\M(R)={\rm dim}\,\mathrm{gr}_\M(R)={\rm dim}\,R>0$. But this implies $s(\M) = 1$, which is a contradiction. Therefore, $${\rm r}_{I}(\M) \geq s(\M) - 1$$ and the result follows from Theorem \ref{CD-thm}. \qed
\end{proof}

\medskip

We close the section with yet another affirmative case of the conjecture.

\begin{Proposition}\label{minmult}
    Conjecture \ref{CD-conject} is true if $R$ has minimal multiplicity.
\end{Proposition}
\begin{proof} For such $R$, the ring ${\rm gr}_{\M}(R)$ is Cohen-Macaulay (see \cite[Theorem 2]{Sally}), hence its depth is equal to ${\rm dim}\,R\geq 2>0$. In this case, as we already know, the conjecture holds true.
\qed
\end{proof}

\section{Application: Regular local rings}

In this section, we provide new characterizations of regular local rings and describe a potential approach to the long-standing Zariski-Lipman conjecture about derivations.

\subsection{Characterizations of regular local rings} Our result in this part is Theorem \ref{charact2} below. We observe that the implication {\rm (a)} $\Rightarrow$ {\rm (d)} recovers \cite[Corollary 3.11]{CD}. Moreover, a crucial fact here (which, as far as we know, is new) is given by {\rm (b)} $\Rightarrow$ {\rm (a)}, i.e., $(R, \M)$ must be regular if $\M$ is generated by a $d$-sequence (see \cite{H1}, \cite{H2}), which in particular solves the problem suggested in \cite[Remark 3.12]{CD}. Finally, the equivalence between assertions (a) and (e) reveals the curious role played by ${\mathfrak d}_3(I)$ in regard to the theory of regular local rings, which can be re-expressed by means of equivalence to the structural assertion (f). 

We recall, for completeness, that a sequence of elements $x_1, \ldots, x_m\in R$ is said to be a $d$-sequence if no element of the sequence lies in the ideal generated by the others and, in addition, there are equalities $0 :_Rx_{1}x_j = 0 :_Rx_j$ for $j=1, \ldots, m$ and $(x_1, \ldots, x_i) :_Rx_{i+1}x_j = (x_1, \ldots, x_i) :_Rx_j$ for $i=1, \ldots, m-1$ and $j=i+1, \ldots, m$. This holds, e.g., if $x_1, \ldots, x_m$ is a regular sequence, but there are plenty of examples of $d$-sequences that are not regular.

For the proof, the following two interesting facts will be useful.

\begin{Lemma}{\rm (\cite[Corollary 5.2]{Trung})} \label{d-seq} Let $(R, \M)$ be a local ring with infinite residue field. Then, $\M$ is generated by a $d$-sequence if and only if $\mathrm{reg}_{\mathcal{R}(\M)} \mathcal{R}(\M) = 0$.
    
\end{Lemma}

\begin{Lemma}{\rm (\cite[Theorem 1.1]{Matsuoka})} \label{param} Let $(R, \M)$ be a local ring with ${\rm depth}\,R>0$. If $Q$ is a parameter ideal of $R$ such that $Q^n$ is $\M$-full for some $n\geq 1$, then $R$ is regular.
    
\end{Lemma}

\begin{Theorem}\label{charact2}
     Let $(R,\M,K)$ be either a local ring or a standard graded $K$-algebra, with $K$ infinite and $\mathrm{depth}\,R > 0$. The following assertions are equivalent:
     \begin{itemize}
         \item[(a)] $R$ is regular;
         \item[(b)] $\M$ is generated by a $d$-sequence;
         
         \item[(c)] $s(\M) = 1$ and ${\rm r}_{I}(\M) = 0$, \  for any {\rm (}minimal{\rm )} reduction $I$ of $\M$;
         \item[(d)] ${\mathfrak d}_1(I) = {\mathfrak d}_2(I) = {\mathfrak d}_3(I) = 0$, \ for any {\rm (}minimal{\rm )} reduction $I$ of $\M$;
     \item[(e)] ${\mathfrak d}_3(I) = 0$, \ for any {\rm (}minimal{\rm )} reduction $I$ of $\M$;
     \item[(f)] $\mathcal{R}(\M, I) =  (\mathcal{R}(\M, I)_{\geq 1} :_{\mathcal{R}(R)} \M)(1)$, \ for any {\rm (}minimal{\rm )} reduction $I$ of $\M$.
     \end{itemize}
     \end{Theorem}
     \begin{proof} {\rm (a)} $\Rightarrow$ {\rm (b)} If $R$ is regular, then $\M$ is generated by a regular sequence, which is therefore a
$d$-sequence.

         {\rm (b)} $\Rightarrow$ {\rm (c)} According to \cite[Theorem 2.1(ii)]{Rossi-Dinh-Trung}, we have ${\rm max}\{\mathrm{reg}_{\mathcal{R}(\M)} \mathcal{R}(\M), 1\} \geq s(\M)$. Hence, using Lemma \ref{d-seq}, we obtain $s(\M)=1$. In order to deal with ${\rm r}_{I}(\M)$, consider first the case where the reduction $I$ is minimal. Applying \cite[Theorem 1.3]{Rossi-Dinh-Trung} and \cite[p.\,12]{Rossi-Valla}, we derive $$\mathrm{reg}_{\mathcal{R}(\M)} \mathcal{R}(\M) \geq {\rm r}_{I}(\M),$$ which gives ${\rm r}_{I}(\M)=0$. Now, if $I$ is not minimal, then it necessarily contains a minimal reduction $J$ of $\M$, so that ${\rm r}_{I}(\M)\leq {\rm r}_{J}(\M)=0$.

          {\rm (c)} $\Rightarrow$ {\rm (d)}  This follows readily from the fact (proved in \cite[Theorem 3.4]{CD}) that $${\mathfrak d}_3(I) \leq \mathrm{max}\{{\rm r}_{I}(\M), s(\M) - 1\}.$$ Now, by \cite[Proposition 2.2]{CD}, the vanishing of ${\mathfrak d}_3(I)$ forces that of  ${\mathfrak d}_1(I)$  and ${\mathfrak d}_2(I)$. 

 {\rm (d)} $\Rightarrow$ {\rm (e)} This is obvious.

          {\rm (e)} $\Rightarrow$ {\rm (a)} In the present setting, $\M$ admits a reduction $Q$ which is a parameter ideal (see, e.g., \cite[Exercise 8.11(ii)]{H-S}). As ${\mathfrak d}_3(Q)= {\mathfrak d}_1(Q)$, we obtain $${\mathfrak d}_1(Q) = 0$$ and consequently $Q$ is $\M$-full. We are now in a position to apply Lemma \ref{param} with $n=1$ to conclude  that $R$ is regular.  

          It remains to prove that assertions (a) and (f) are equivalent. To this end, simply note that (f) holds if and only if $(\mathcal{R}(\M, I))_k =  (\mathcal{R}(\M, I)_{k \geq 1} :_{\mathcal{R}(R)} \M)(1)_k$ for all $k \geq 0$ and any {\rm (}minimal{\rm )} reduction $I$ of $\M$, which means $$I\M^{k} = I\M^{k + 1} : \M \quad \mbox{for\, all} \quad k \geq 0.$$ This is, by definition, tantamount to ${\mathfrak d}_3(I) = 0$, which as we have shown above is equivalent to $R$ being regular.
          \qed
     \end{proof}

\medskip 

The equivalence {\rm (a)} $\Leftrightarrow$ {\rm (e)} of Theorem \ref{charact2} leaves the following natural problem.

\begin{Question}\rm
 Let $(R, \M , K)$ be either a local ring or a standard graded $K$-algebra, with $K$ infinite and $\mathrm{depth}\,R > 0$. What kind of rings could be characterized
by the property that ${\mathfrak d}_3(I) \leq 1$ for all minimal reductions $I$ of $\M$?
\end{Question}

Though quite vaguely and having no indication of proof, the authors suspect that the answer to the question should rely in the nature of the singular locus of $R$, for instance, the situation where $R$ has at most an isolated singularity at the origin (the case of rational singularities should not be ruled out either).

\subsection{Potential approach to a classical conjecture}

For a field $K$ and a $K$-algebra $R$, we write as usual ${\rm Der}_K(R)$  for the module of $K$-derivations of $R$, i.e., the additive maps $D\colon R\to R$ that vanish on $K$ and satisfy Leibniz rule:  $D(\alpha \beta)=\alpha D(\beta)+\beta D(\alpha)$ for all $\alpha, \beta \in R$. Now assume that $R$ is a positive-dimensional local ring which is either
\begin{equation}\label{ZL}
{K}[x_1, \ldots, x_m]_{\mathfrak q}/I \ \ \ ({\mathfrak q}\in {\rm Spec}\,{K}[x_1, \ldots, x_m])\quad \mbox{or} \quad {K}[\![x_1, \ldots, x_m]\!]/I,\end{equation} with $I$ a proper radical ideal and $x_1, \ldots, x_m$ indeterminates over a field ${K}$ of characteristic zero. In this setting, there is the following long-held classical problem.

\begin{Conjecture}\label{ZLC} {\rm (Zariski-Lipman)} Let $R$ be as in {\rm (\ref{ZL})}. If ${\rm Der}_K(R)$ is free, then $R$ is regular.

\end{Conjecture}

This problem has an interesting long history, featuring in particular a strong geometric counterpart, and remains open in some cases. For  details and references, see \cite[Section 2]{H} (in particular, see \cite[Theorem 2.3]{H} for a simple proof of the Zariski-Lipman conjecture in the graded case). Additionally,  \cite[Section 4]{Cleto} shows a relation between this conjecture and the $\M$-full property of ideals in the (open) two-dimensional local case, which makes it somewhat natural to expect further connections.

We point out that, under the hypotheses of the conjecture, $R$ must be a (normal) domain, and we can write $${\rm Der}_K(R)=\bigoplus_{1\leq j\leq t}RD_j ~~ \cong R^t, \quad t={\rm dim}\,R,$$ for some free basis $\{D_j\}$  consisting of precisely $t$ derivations. Conversely, if ${\rm Der}_K(R)$ admits a free basis $\{D_1, \ldots, D_s\}$, then necessarily $s=t$.

\begin{Remark}\rm Let $(R, \M)$ be as in {\rm (\ref{ZL})}. If, as above, the $R$-module  ${\rm Der}_K(R)$ is free, then our guess is that, for each minimal reduction $I$ of $\M$, there exist $\alpha_j^{(I)}, \beta_j^{(I)}\in R$, $j=1, \ldots, t$, such that the element given by 
$${\ell}^{(I)}=\sum_{j=1}^t\beta_j^{(I)}D_j(\alpha_j^{(I)})$$ satisfies ${\ell}^{(I)}\in \M \setminus {\M}^2$ and $I{\M}^{k+1}\colon ({\ell}^{(I)})=I{\M}^k$ for all  $k\geq 0$. This would confirm the validity of Conjecture \ref{ZLC}, because such conditions yield $I{\M}^k$ to be $\M$-full for all $k\geq 0$, which means $\mathfrak{d}_1(I)=0$ and hence, as we already know, $\mathfrak{d}_3(I)=0$. Now, Theorem \ref{charact2} ensures that $R$ is regular.
\end{Remark}

\bigskip

\noindent{\bf Acknowledgements.} The first-named author was partly supported by the Grant JDC2023-051705-I funded by
MICIU/AEI/10.13039/501100011033 and by the FSE+. The second-named author was partially supported by CNPq (grants 406377/2021-9 and 313357/2023-4). The authors are grateful to the referee for a number of helpful comments and suggestions, and for pointing out an inaccuracy in an earlier proof of Theorem \ref{antonino1.3}.

\end{document}